\documentclass[11pt]{elsarticle}
\usepackage{amsmath,amssymb,amsthm}
\usepackage{hyperref}

\theoremstyle{plain}
\newtheorem{theorem}{Theorem}[section]

\newtheorem{proposition}[theorem]{Proposition}
\newtheorem{corollary}[theorem]{Corollary}

\theoremstyle{definition}

\newtheorem{remark}[theorem]{Remark}

\newcommand{\ua}{\mathord{\uparrow}}
\newcommand{\da}{\mathord{\downarrow}}
\newcommand{\rom}[1]{\rm{\uppercase\expandafter{\romannumeral #1}}}

\newcommand{\set}[2]{\{#1\mid#2\}}
\newcommand{\oneset}[1]{\{#1\}}

\begin{document}

\begin{frontmatter}




\title{Core-compactness of Smyth powerspaces}

\author[1]{Zhenchao Lyu}
\address[1]{LSV, ENS Paris-Saclay, CNRS, Universit\'e Paris-Saclay, Cachan, 94230, France}
\ead{lyu@lsv.fr}

\author[2]{Xiaodong Jia}
\address[2]{College of Mathematics and Econometrics, Hunan University, Changsha, 410082, China}
\ead{jia.xiaodong@yahoo.com}

\begin{abstract}
We prove that the Smyth powerspace $Q(X)$ of a topological space $X$ is core-compact if and only if $X$ is locally compact. As a straightforward consequence we obtain that the Smyth powerspace construction does not preserve core-compactness generally.
\end{abstract}

\end{frontmatter}

\section{Introduction}

Given a topological space $X$,
the \emph{Smyth powerspace} $Q(X)$ is the set of compact saturated subsets of $X$ with the \emph{upper Vietoris topology}. In domain theory, the Smyth
powerspace coincides with the Smyth powerdomain for continuous domains with the Scott topology, where the latter construction is used in modelling non-deterministic computation, see for example~\cite{smyth83a, goubault15}.  The Smyth powerspace construction has many nice properties and useful applications. For example, it was proved by Schalk~\cite{schalk93a}, Heckmann and Keimel~\cite{heckmann13} that a space is sober if and only if its Smyth powerspace is sober. 
Xu, Xi and Zhao~\cite{xu19} proved that a similar result holds for well-filtered spaces. That is, a space is well-filtered if and only if its Smyth powerspace is well-filtered.  In the same paper,  the Smyth powerspace construction was heavily employed in giving a spatial frame which is not sober in its Scott topology. 

In this note, we consider another important topological property, core-compactness, and investigate whether it can be preserved by the Smyth powerspace construction. A topological space is \emph{core-compact} if and only if the lattice of its open subsets (under set inclusion) is a continuous lattice in the sense of domain theory. Core-compact spaces are of great importance in topology and domain theory since these spaces are precisely exponentiable objects in the category of $T_0$ topological spaces and continuous functions. We prove that for a topological space $X$, its Smyth powerspace $Q(X)$ is core-compact if and only if $Q(X)$ is locally compact if and only if $X$ is locally compact. Since there exists core-compact spaces which are not locally compact~\cite{hofmann78}, it follows that core-compactness cannot be preserved by the Smyth powerspace construction in general.

\section{Preliminaries}
We refer to~\cite{gierz03, abramsky94, goubault13a} for the standard definitions and notations of order theory, topology and domain theory.
For a topological space $X$,  we use $O(X)$ to denote the lattice of open subsets of~$X$.
A topological space is a \emph{c-space} if  for any $x\in X$ and any open neighbourhood $U$ of $x$, there is a point $y\in U$ such that $x\in \text{int}({\uparrow}y)$, where 
the symbol $\ua$ is the saturation operator. For a subset $A$ of space $X$, $\ua A$ is the the intersection of all open neighbourhood of $A$ and called the saturation of $A$. A set~$A$ is called \emph{saturated} if and only if $A=\ua A$. A set $A$ is compact if and only if its saturation $\ua A$ is compact. 
For a topological space $X$, we denote the set of all compact saturated sets of $X$ by $Q(X)$. We
consider the \emph{upper Vietoris topology} on $Q(X)$, generated by the sets $\Box U=\{K\in Q(X): K\subseteq U\}$, where $U$ ranges over the open subsets of $X$. One sees that $\Box U$'s form a base of the upper Vietoris topology since $\Box U\cap \Box V=\Box (U\cap V)$ for open sets~$U, V$. For a compact saturated set $G$, we use $\ua_v G$ to denote the saturation of the singleton $\oneset{G}$ with respect to the upper Vietoris topology on $Q(X)$. Note that $\ua_v  G=\set{K\in Q(X)}{ K\subseteq G }$. 

Let $P$ be a poset and $B$ be a subset of $P$, we say that $B$ is a \emph{basis} of $P$, if $a=\bigvee (\da a \cap B)$ for all $a\in P$, where $\da a$ is the set of all elements that are below $a$. For a subset $A$ of~$P$, we fix $\da A= \bigcup_{a\in A}\da a$. 

Let $L$ be a complete lattice, we define the \emph{way-way-below relation} $\lll$ on $L$ by $x\lll y$ if for any $A\subseteq L$ with $y\leq\bigvee A$, there is $a\in A$ such that $x\leq a$. We call $L$ \emph{prime-continuous} if for any $x\in L$, $x=\bigvee\{y\in L: y\lll x\}$ holds. 

Every prime-continuous complete lattice is a continuous lattice. The following proposition provides a criteria for a continuous lattice to be prime-continuous. 

\begin{proposition}\label{bac}
Let $L$ be a continuous lattice with a basis $B$. If for any $b\in B$ and finite $F\subseteq B$, $b\leq \bigvee F$ implies that $b\in \da F$, then $L$ is prime-continuous. 
\end{proposition}	
\begin{proof}
Give $x\in L$ and $b\in B$, we prove that $b\ll x$ if and only if $b\lll x$. The ``if " direction is obvious. For the converse we assume that $b\ll x$ and let $A$ be any subset of~$L$ with $x\leq \bigvee A$. Since $B$ is a basis of $L$, we know that $\bigvee A = \bigvee (\da A\cap B)$. This means that we can find a finite subset $F$ of $\da A\cap B$ such that $b\leq \bigvee F$ as $b\ll x$. Notice that $b\in B$ and $F\subseteq B$, by assumption there exists an element $f\in F\subseteq \da A\cap B$ such that $b\leq f$. Hence $b$ is below some point of $A$, and this implies that $b\lll x$.
\end{proof}

The following result about c-spaces and prime-continuity is well-known in domain theory, and the proof can be found in~\cite{ho16}, for example. 

\begin{theorem}\label{ptt}
Let $X$ be a topological space.
Then $X$ is a c-space iff $O(X)$ is prime-continuous.	
\end{theorem}

\section{Main results}

We arrive at the main result of this note.

\begin{theorem}
Let $X$ be a topological space. The following statements are equivalent:
\begin{enumerate}
	\item $X$ is locally compact;
	\item $Q(X)$ is a c-space;
	\item $Q(X)$ is locally compact;
	\item $Q(X)$ is core-compact.
\end{enumerate}  	
\end{theorem}

\begin{remark}
 The equivalence between $(1)$ and $(2)$ is folklore among domain theorists and the proof can also be found in \cite{ho18a}.
\end{remark}

\begin{proof}

 $(1)\Rightarrow (2)$: Let $U$ be an open set of~$X$ and $K$ be a compact saturated set in~$\Box U$. This means that $K\subseteq U$. Since $X$ is locally compact and $K$ is compact, we can find an open set $V$ and a compact saturated set G such that $K\subseteq V\subseteq G\subseteq U$.  It follows that $K\in \Box V \subseteq \ua_v G \subseteq \Box U$. This implies that $Q(X)$ is a c-space.

 $(2)\Rightarrow (1)$: For any $x\in X$ and any open neighbourhood $U$ of $x$, it is clear that  ${\uparrow}x\in\Box U$. Since $Q(X)$ is a c-space, there are $K\in Q(X)$ and $V\in O(X)$ such that ${\uparrow}x\in \Box V\subseteq \ua_v K\subseteq\Box U$. It follows that $x\in V\subseteq K\subseteq U$. Therefore $X$ is locally compact. 
 
$(2)\Rightarrow (3)$: Obvious. 
 
$(3)\Rightarrow (4)$: Obvious.
 
 $(4)\Rightarrow(2)$: In light of Theorem~\ref{ptt}, we prove $Q(X)$ is a c-space by showing that its open set lattice $O(Q(X))$ is prime-continuous.  Since $Q(X)$ is core-compact, $O(Q(X))$ is a continuous lattice. Moreover, the set $\set{\Box U}{U\in O(X)}$ is a base of the upper Vietoris topology on $Q(X)$, then it is a basis of the continuous lattice $O(Q(X))$. By Proposition~\ref{bac}, without loss of generality, we only need to prove that $\Box U\subseteq \Box V$ or $\Box U\subseteq \Box W$ whenever $\Box U\subseteq \Box V \cup \Box W$, where $U, V, W$ are opens in $X$. This is just a small variant of~\cite[Lemma 4.2]{goubault11}; we speak the proof in full, nevertheless. Assume this is not true. Then we can find compact sets $K_i \subseteq U, i=1, 2$,  such that $K_1\not\subseteq V$ and $K_2\not\subseteq W$.  So the union $K_1\cup K_2$, which is compact saturated, is not in $\Box V \cup \Box W$. However this is impossible since  $K_1\cup K_2 \subseteq U$ and $\Box U\subseteq \Box V \cup \Box W$.
 \end{proof}

The following also appears as Exercise V-5.25 of \cite{gierz03}.
\begin{theorem}{\rm\cite{hofmann78}}
There exists a core-compact topological space which is not locally compact.	
\end{theorem}

Combining the above theorems, we get our final result.
\begin{corollary}
Let X be a core-compact space but not locally compact. Then Q(X) is not core-compact.	\hfill $\Box$
\end{corollary}

\section*{Acknowledgement}
This research was partially supported by Labex DigiCosme (project ANR-11-LABEX0045-DIGICOSME) operated by ANR as part of the program ``Investissement d'Avenir" Idex
Paris-Saclay (ANR-11-IDEX-0003-02). 
We would also like to thank Jean Goubault-Larrecq for helpful discussions and pointing us to the result~\cite[Lemma 4.2]{goubault11}. The second author would like to thank Prof. Qingguo Li for hosting him as a visiting scholar at Hunan University from May to July 2019.


\begin{thebibliography}{10}

\bibitem{abramsky94}
S.~Abramsky and A.~Jung.
\newblock Domain theory.
\newblock In S.~Abramsky, D.~M. Gabbay, and T.~S.~E. Maibaum, editors, {\em
  Semantic Structures}, volume~3 of {\em Handbook of Logic in Computer
  Science}, pages 1--168. Clarendon Press, 1994.

\bibitem{gierz03}
G.~Gierz, K.~H. Hofmann, K.~Keimel, J.~D. Lawson, M.~Mislove, and D.~S. Scott.
\newblock {\em Continuous Lattices and Domains}, volume~93 of {\em Encyclopedia
  of Mathematics and its Applications}.
\newblock Cambridge University Press, 2003.

\bibitem{goubault13a}
J.~Goubault-Larrecq.
\newblock {\em Non-Hausdorff Topology and Domain Theory}, volume~22 of {\em New
  Mathematical Monographs}.
\newblock Cambridge University Press, 2013.

\bibitem{goubault15}
J.~Goubault-Larrecq.
\newblock Full abstraction for non-deterministic and probabilistic extensions
  of~{PCF}~{I} --- the angelic cases.
\newblock {\em Journal of Logic and Algebraic Methods in Programming},
  84(1):155 -- 184, 2015.
\bibitem{goubault11}J. Goubault-Larrecq, K. Keimel. \newblock Choquet-Kendall-Matheron theorems for non-Hausdorff space.
\newblock {\em Mathematical Structures in Computer Science}, 21:511--561, 2011.
\bibitem{heckmann13}
R.~Heckmann and K.~Keimel.
\newblock Quasicontinuous domains and the {S}myth powerdomain.
\newblock In D.~Kozen and M.~Mislove, editors, {\em Proceedings of the 29th
  Conference on the Mathematical Foundations of Programming Semantics}, volume
  298 of {\em Electronic Notes in Theoretical Computer Science}, pages
  215--232. Elsevier Science Publishers {B.V.}, 2013.

\bibitem{ho18a}
W.~Ho.
\newblock A domain-theoretic silk road.
\newblock Talk at An Intersection of Neighbourhoods ``Achim Jung Fest: A
  commemoration in honour of his 60th birthday", September 2018.
  
  \bibitem{ho16}
W.~Ho, A.~Jung, and D.~Zhao.
\newblock  Join-continuity+ hypercontinuity= prime continuity.
\newblock Available at \url{https://arxiv.org/abs/1607.01886}, 2016.
 

\bibitem{hofmann78}
K.~H. Hofmann and J.~Lawson.
\newblock The spectral theory of distributive continuous lattices.
\newblock {\em Transitions of the American Mathematical Society}, 246:285--310,
  1978.

\bibitem{schalk93a}
A.~Schalk.
\newblock {\em Algebras for Generalized Power Constructions}.
\newblock Doctoral thesis, Technische Hochschule Darmstadt, 1993.
\newblock 166~pp.

\bibitem{smyth83a}
M.~B. Smyth.
\newblock Powerdomains and predicate transformers: a topological view.
\newblock In J.~Diaz, editor, {\em Automata, Languages and Programming}, volume
  154 of {\em Lecture Notes in Computer Science}, pages 662--675. Springer
  Verlag, 1983.

\bibitem{xu19}
X.~Xu, X.~Xi, and D.~Zhao.
\newblock A complete Heyting algebra whose scott space is non-sober.
\newblock Available at \url{https://arxiv.org/abs/1903.00615}, 2019.

\end{thebibliography}
\end{document}